\documentclass{article}

\usepackage{amsmath}        
\usepackage{amssymb}        
\usepackage{mathtools}      
\usepackage{hyperref}       
\usepackage{geometry}       
\usepackage{stmaryrd}
\usepackage{todonotes}
\usepackage{amsthm}         
\usepackage{xcolor}         

\theoremstyle{definition}
\newtheorem{theorem}{Theorem}[section]
\newtheorem*{theorem*}{Theorem}
\newtheorem{definition}[theorem]{Definition}
\newtheorem*{definition*}{Definition}
\newtheorem{proposition}[theorem]{Proposition}

\newtheorem*{proposition*}{Proposition}
\newtheorem{lemma}[theorem]{Lemma}
\newtheorem*{lemma*}{Lemma}
\newtheorem{corollary}[theorem]{Corollary}
\newtheorem*{corollary*}{Corollary}

\newtheorem{example}[theorem]{Example}
\newtheorem*{example*}{Example}
\newtheorem{remark}[theorem]{Remark}
\newtheorem*{remark*}{Remark}

\newtheorem*{notation*}{Notation}

\hypersetup{
    colorlinks=true,
    linkcolor=blue,
    filecolor=magenta,
    urlcolor=cyan,
    citecolor=red, 
    pdftitle={Regularity and Convergence Properties of Finite Free Convolutions},
    pdfauthor={Katsunori Fujie},
}

\newcommand{\calP}{\mathcal{P}}
\DeclareMathOperator{\diff}{\partial} 
\newcommand{\cdf}{\mathcal{F}}
\DeclareMathOperator{\Dil}{Dil}
\DeclareMathOperator{\Shi}{Shi}
\newcommand{\calM}{\mathcal{M}}
\newcommand{\meas}[1]{\mu\llbracket #1 \rrbracket} 

\newcommand{\weakto}{\xrightarrow{w}} 

\newcommand{\coef}{\tilde{e}}
\newcommand{\N}{\mathbb{N}}

\newcommand{\R}{\mathbb{R}}
\newcommand{\C}{\mathbb{C}}
\newcommand{\longimplies}{\Longrightarrow}
\renewcommand{\epsilon}{\varepsilon}

\numberwithin{equation}{section} 

\begin{document}

\title{Regularity and Convergence Properties of Finite Free Convolutions}
\author{Katsunori Fujie}
\date{\today} 

\maketitle

\begin{abstract}
Finite free convolutions, $\boxplus_d$ and $\boxtimes_d$, are binary operations on polynomials of degree $d$ that are central to finite free probability, a developing field at the intersection of free probability and the geometry of polynomials.
Motivated by established regularities in free probability, this paper investigates analogous regularities for finite free convolutions.
Key findings include triangle inequalities for these convolutions and necessary and sufficient conditions regarding atoms of probability measures.
Applications of these results include proving the weak convergence of $\boxplus_d$ and $\boxtimes_d$ to their infinite counterparts $\boxplus$ and $\boxtimes$ as $d \to \infty$, without compactness assumptions. Furthermore, this weak convergence is strengthened to convergence in Kolmogorov distance.
\end{abstract}

\tableofcontents

\section{Introduction}

\subsection{Background and motivation for the research}

Finite free probability is a branch that has been rapidly developing since the celebrated foundational work by Marcus, Spielman, and Srivastava \cite{MSS22} (see also \cite{Marcus}).
Their work was originally motivated by longstanding problems in operator algebra and graph theory; they provided an affirmative answer to the Kadison--Singer problem and constructed non-trivial Ramanujan graphs \cite{MSS1, MSS2}.
They primarily utilized techniques known from geometry of polynomials, such as real-rootedness and interlacing properties (see Definition \ref{def:interlacing} and the explanation below). Their main contribution was to highlight the connections among free probability, random matrix theory, and polynomial convolutions defined almost 100 years ago by Szeg\H{o} and Walsh, which are termed finite free additive and multiplicative convolutions, respectively, in this paper.

For $d \in \N$, let $\calP_d(\C)$ denote the set of monic polynomials of degree $d$.
We also use $\calP_d(\R)$ to denote the subset of $\calP_d(\C)$ having only real roots.
For $p \in \calP_d(\C)$, we express $p$ as
\begin{equation*}
p(x) = \sum_{k=0}^d x^{d-k} (-1)^k \binom{d}{k} \coef_k^{(d)} (p).
\end{equation*}
We adopt this seemingly unfamiliar notation as we believe it is optimal for stating the definitions and various results.
Given polynomials $p, q \in \calP_d(\C)$, \textit{the finite free additive and multiplicative convolutions} are respectively defined as follows:
\begin{align}
p \boxplus_d q (x) &= \sum_{k=0}^d x^{d-k} (-1)^k \binom{d}{k} \sum_{i=0}^k \binom{k}{i} \coef_i^{(d)} (p) \coef_{k-i}^{(d)} (q); \label{eq:finite-free-additive}\\
p \boxtimes_d q (x) &= \sum_{k=0}^d x^{d-k} (-1)^k \binom{d}{k} \coef_k^{(d)} (p) \coef_k^{(d)} (q). \label{eq:finite-free-multiplicative}
\end{align}

The main interest of this paper lies in these two operations. The first, \eqref{eq:finite-free-additive}, was previously defined and studied by Walsh \cite{Walsh}, while the second, \eqref{eq:finite-free-multiplicative}, was essentially defined by Szegő \cite{Sz22} and is known as the Schur–Szegő composition (see also \cite[Chapter 4]{Mar66}). For the reader's convenience, their basic properties are summarized in Section \ref{sec:finitefree}. The foundational work of Marcus, Spielman, and Srivastava establishes the following connection \cite[Theorems 1.2 and 1.5]{MSS22}: let $A$ and $B$ be $d \times d$ normal matrices and $U$ be a $d \times d$ Haar-distributed unitary matrix (i.e., a random unitary matrix uniformly distributed on the unitary group of degree $d$). If $p(x) = \det(x I - A)$ and $q(x) = \det(x I - B)$, then
\[
p \boxplus_d q (x) = E_U [\det(x I - A - U B U^*)] \qquad \text{and} \qquad p \boxtimes_d q (x) = E_U [\det(x I - A U B U^* A)],
\]
which indicates their connection to free probability.

Free probability, pioneered by Voiculescu in the 1980s, provides a powerful framework for understanding non-commutative random variables, particularly in the context of operator algebras and random matrix theory \cite{V85,V86,V87,Voi91} (see also the standard textbooks \cite{MS, NS, VDN92} and references therein). At its heart lies the concept of ``freeness,'' an analogue of classical independence for non-commutative variables, and many analogues to concepts in classical probability have been established.
For instance, given probability measures $\mu, \nu$ on $\R$, \textit{the free additive convolution} $\mu \boxplus \nu$ is defined as the distribution of $X + Y$, where $X$ and $Y$ are freely independent non-commutative random variables distributed as $\mu$ and $\nu$, respectively. Additionally, if the support of $\nu$ is contained in $[0,\infty)$ (equivalently, $Y \ge 0$), \textit{the free multiplicative convolution} $\mu \boxtimes \nu$ is defined as the distribution of $\sqrt{Y}X\sqrt{Y}$.
It then becomes natural to define and study, for example, the law of large numbers, central limit theorems, free infinite divisibility, and the Lévy--Khintchine representation, among others.

This theory has achieved remarkable success in describing the asymptotic behavior of large $N\times N$ random matrices as $N$ approaches infinity.
For instance, the limiting eigenvalue distribution of the sum of two independent random matrices, whose distributions are invariant under unitary action, is described by the free additive convolution of their individual limiting distributions.
This phenomenon is well known as the asymptotic freeness of random matrices.
The relation between the operations $\boxplus_d$ and $\boxtimes_d$ and random matrices, discovered by Marcus, Spielman, and Srivastava, suggests a fundamental connection to the free convolutions $\boxplus$ and $\boxtimes$ as $d \to \infty$.

Indeed, many free probabilistic analogues hold in finite free probability, including the law of large numbers and the central limit theorem \cite{Marcus}, finite free cumulants \cite{AP}, and finite $S$-transforms \cite{AFPU}, among others (see Section \ref{sec:finitefree}).
The most basic of these are the approximations of $\boxplus_d$ and $\boxtimes_d$ to $\boxplus$ and $\boxtimes$ as $d \to \infty$.
To explain this briefly: given $p(x) = \prod_{i=1}^d (x-\lambda_i) \in \calP_d(\R)$, the empirical root distribution is defined as $\meas{p} = \frac{1}{d} \sum_{i=1}^d \delta_{\lambda_i}$.
For sequences $p_d, q_d \in \calP_d(\R)$ such that $\meas{p_d} \weakto \mu$ and $\meas{q_d} \weakto \nu$, it holds that $\meas{p_d \boxplus_d q_d} \weakto \mu \boxplus \nu$; a similar convergence holds for $\boxtimes_d$ to $\boxtimes$ under the additional condition $q_d \in \calP_d(\R_{\ge 0})$.
To date, however, proofs have essentially relied on moment convergence methods, which necessitate conditions on the supports of the measures.
One result of this paper (Theorem \ref{thm:main3}) is to demonstrate that this condition on the supports can be removed.

Generally, studying the properties of free convolutions is difficult.
Historically, definitions were initially given only for measures with compact support.
More general definitions were provided in the 1990s \cite{BV93, Maa92}.
However, various regularity properties are now known due to the efforts of many researchers in this area. These include: the monotone property for order \cite[Propositions 4.15 and 4.16]{BV93} and distances \cite[Proposition 4.13 and 4.14]{BV93}; properties of atoms and absolute continuity with respect to Lebesgue measure for $\boxplus$ \cite[Theorem 7.4]{BV98}, \cite[Theorem 4.1]{Bel08}, and for $\boxtimes$ \cite[Corollary 6.6]{ACSY}, \cite[Theorem 3.1]{Ji}, \cite[Theorem 1.9]{AHK}; among others.
These properties are summarized in Section \ref{sec:FreeProbability}.

In this paper, motivated by known regularities in free probability, we aim to contribute to the research on the regularity of finite free convolutions.
Usually, investigating the properties of free convolutions requires analytical approaches involving tools such as the $R$-transform, $S$-transform, and their associated free subordination functions. However, in finite free probability, with the exception of the finite $S$-transform for $\calP_d(\R_{\ge 0})$ defined in \cite{AFPU}, many challenges remain in developing analytical approaches (see \cite{Campbell, JKM, Marcus, MSS22}).
Therefore, this paper employs a more elementary approach, similar to that taken in \cite{BV93}, to establish regularity properties of finite free convolutions.
Our main tools are the partial order defined on $\mu \in \calM(\R)$ and approximations using cut-up $\mu|^a$ and cut-down $\mu|_a$ measures.

\subsection{Main results}

To save space, this section primarily presents results for the finite free additive convolution $\boxplus_d$; similar results hold for the finite free multiplicative convolution $\boxtimes_d$ (see Propositions \ref{prop:boxtimes-atom}, \ref{prop:atoms-finite-free-multiplicative}, and \ref{prop:simple-finite-free-multiplicative}; Theorems \ref{thm:monotonicity-finite-free-multiplicative}, \ref{thm:convergence_multiplicative}, and \ref{thm:convergence_mixed}).

The first result concerns atoms in finite free convolutions.
Let $p \in \calP_d(\C)$ (resp. $q \in \calP_d(\C)$) be a monic polynomial of degree $d$ having a root $\alpha$ with multiplicity $m^p_\alpha$ (resp. having a root $\beta$ with multiplicity $m^q_\beta$).
If $m^p_\alpha + m^q_\beta - d > 0$, then $p \boxplus_d q$ has a trivial root $\alpha+\beta$ with multiplicity $m^p_\alpha + m^q_\beta - d$.
We call such a triplet $(\alpha, \beta, \alpha + \beta)$ an atom triplet of $(p, q, p \boxplus_d q)$.
The other roots of $p \boxplus_d q$ are said to be non-trivial.
Research on the finite free multiplicative convolution (Schur--Szeg\H{o} composition) was previously conducted by Kostov and Shapiro in 2006 \cite{KS}.
Inspired by their work, we obtained the following results.
In the statements that follow, for $p \in \calP_d(\R)$, we use the abbreviation $\cdf_p$ to denote the cumulative distribution function of the empirical root distribution $\meas{p}$; that is, $\cdf_p = \cdf_{\meas{p}}$.
\begin{theorem} \label{thm:main1}
Let $p, q \in \calP_d(\R)$ be monic polynomials of degree $d$ having only real roots.
\begin{enumerate}
\item Every non-trivial root of $p \boxplus_d q$ is simple.
\item For an atom triplet $(\alpha, \beta, \alpha+\beta)$ of $(p, q, p \boxplus_d q)$, we have
\[
\meas{p \boxplus_d q}(\{\alpha+\beta\}) = \meas{p}(\{\alpha\}) + \meas{q}(\{\beta\}) -1
\]
and also
\[
\cdf_{p \boxplus_d q}(\alpha+\beta) = \cdf_{p}(\alpha) + \cdf_{q}(\beta) - 1.
\]

\end{enumerate}
\end{theorem}

Our second main result establishes a monotone property for distances under finite free convolutions.
We consider two distances defined on probability measures $\mu, \nu$ on $\R$: the Kolmogorov distance $d_K(\mu, \nu)$ and the Lévy distance $d_L(\mu, \nu)$, which are defined in the usual manner (see Definition \ref{def:Distance.on.measures}).
For polynomials $p, q \in \calP_d(\R)$, we adopt the abbreviations $d_K(p,q) := d_K(\meas{p}, \meas{q})$ and $d_L(p,q) := d_L(\meas{p}, \meas{q})$.

\begin{theorem} \label{thm:main2}
Let $p, q, r\in \calP_d (\R)$ be monic polynomials of degree $d$ having only real roots.
\begin{enumerate}
\item $d_K (p \boxplus_d r, q \boxplus_d r) \le d_K (p, q)$.
\item $d_L (p \boxplus_d r, q \boxplus_d r) \le d_L (p, q)$.
\end{enumerate}
\end{theorem}

The preceding relation concerning the monotonicity of distances (Theorem \ref{thm:main2}) is instrumental in extending convergence arguments to the setting of free convolutions.

\begin{theorem} \label{thm:main3}
Let $p_d, q_d \in \calP_d (\R)$ such that their empirical root distributions $\meas{p_d}$ and $\meas{q_d}$ weakly converge to $\mu, \nu \in \calM (\R)$ as $d \to \infty$, respectively.
Then $\meas{p_d \boxplus_d q_d}$ weakly converges to $\mu \boxplus \nu$ as $d \to \infty$.
If $\lim_{d\to\infty} d_K(\meas{p_d}, \mu) = 0$ and $\lim_{d\to\infty} d_K(\meas{q_d}, \nu) = 0$, then $\lim_{d\to\infty} d_K(\meas{p_d \boxplus_d q_d}, \mu \boxplus \nu) = 0$.
\end{theorem}
This theorem demonstrates that the finite free additive convolution $\boxplus_d$ effectively approximates the free additive convolution $\boxplus$ as $d \to \infty$.

\subsection{Organization of the paper}
This paper is organized into four sections.
Section \ref{sec:preliminary} introduces fundamental concepts and preliminary results concerning both free and finite free probability.
Section \ref{sec:atoms} is devoted to studying the roots of finite free additive and multiplicative convolutions of polynomials, culminating in the proof of Theorem \ref{thm:main1}.
Finally, Section \ref{sec:convergence}, focuses on convergence properties and the proofs of Theorems \ref{thm:main2} and \ref{thm:main3}.

\section{Preliminary} \label{sec:preliminary}

\subsection{Probability measures} \label{sec:Probability}

We use the symbol $\calM(\C)$ to denote the family of all Borel probability measures on $\C$.
When we want to specify that the support of the measure is contained in a subset $K \subset \C$, we use the notation $\calM(K)$.
For instance, $\calM(\R)$ denotes the set of Borel probability measures on the real line.
For a sequence of probability measures $\{\mu_n\}_{n \in \N} \subset \calM(\R)$ and $\mu \in \calM(\R)$, the weak convergence $\mu_n \weakto \mu$ is defined by
\[
\lim_{n\to\infty}\int_\R f(x) \mu_n(dx) = \int_\R f(x) \mu(dx),
\]
for all bounded continuous functions $f: \R \to \R$.
Given $\mu \in \calM(\R)$, the cumulative distribution function (CDF) $\cdf_\mu$ is defined by
\[
\cdf_\mu (x) := \mu((-\infty, x]),
\]
for all $x \in \R$.
It is well known that the weak convergence of $\mu_n$ to $\mu$ is equivalent to the point-wise convergence of their CDFs $\cdf_{\mu_n}$ to $\cdf_\mu$ at all continuity points of $\cdf_\mu$.

\begin{definition}[Distances on measures] \label{def:Distance.on.measures}
Let $\mu$ and $\nu$ be probability measures on $\R$.
\begin{enumerate}
\item The Kolmogorov distance of $\mu$ and $\nu$ is defined as
\[
d_K (\mu, \nu) := \inf\{\epsilon >0 \mid \cdf_\mu (x) - \epsilon \le \cdf_\nu (x) \le \cdf_\mu (x) + \epsilon \quad (\text{for all } x \in \R)\}.
\]
\item The Lévy distance of $\mu$ and $\nu$ is defined as
\[
d_L (\mu, \nu) := \inf\{\epsilon >0 \mid \cdf_\mu (x - \epsilon) - \epsilon \le \cdf_\nu (x) \le \cdf_\mu (x + \epsilon) + \epsilon  \quad (\text{for all } x \in \R)\}.
\]
\end{enumerate}
\end{definition}
Clearly, $0 \le d_L (\mu, \nu) \le d_K (\mu, \nu) \le 1$ by definition.
It is well known that the topology induced by the Lévy distance is the same as the weak topology on probability measures, whereas the Kolmogorov distance induces a stronger topology.
For $\mu, \mu_n \in \calM(\R)$ ($n\in \N$), the convergence $\lim_{n\to\infty} d_K(\mu_n,\mu) = 0$ is equivalent to $\lim_{n\to\infty} \cdf_{\mu_n}(x) = \cdf_{\mu}(x)$ for all $x\in\R$.
Hence, if $\mu_n \weakto \mu$ as $n\to\infty$, and $\lim_{n\to\infty} \cdf_{\mu_n}(x) = \cdf_{\mu}(x)$ for every $x\in\R$ such that $\mu(\{x\})>0$, then $\lim_{n\to\infty} d_K(\mu_n,\mu) = 0$.

\begin{definition}[Transformations on measures] \label{def:transformation.on.measures}
Let $\mu \in \calM (\C)$ and $c\in \C$.
\begin{enumerate}
\item The shift of $\mu$ by $c$, denoted $\Shi_c \mu$, is the measure satisfying
\[
(\Shi_c \mu) (B) = \mu(\{x-c \mid x \in B\}) \qquad \text{for every Borel set } B \subset \C.
\]
\item For $c \neq 0$, the dilation of $\mu$ by $c$, denoted $\Dil_c \mu$, is the measure satisfying
\[
(\Dil_c \mu) (B) = \mu(\{x/c \mid x \in B\}) \qquad \text{for every Borel set } B \subset \C.
\]
If $c = 0$, we adopt the convention $\Dil_0 \mu = \delta_0$.
\item The reflected measure $\widehat{\mu} = \Dil_{-1} \mu \in \calM(\C)$ is given by
\[
\widehat{\mu} (B) = \mu(\{-x \mid x \in B\}) \qquad \text{for every Borel set } B \subset \C.
\]
\item For $\mu \in \calM(\C\setminus\{0\})$, the reversed measure $\mu^{\langle -1 \rangle}$ is given by
\[
\mu^{\langle -1 \rangle} (B) = \mu(\{1/x \mid x \in B\}) \qquad \text{for every Borel set } B \subset \C\setminus\{0\}.
\]
\end{enumerate}
\end{definition}
It is clear that $\cdf_{\Shi_c \mu}(x) = \cdf_\mu (x-c)$ for $\mu \in \calM(\R)$ and $c \in \R$.
Hence, $d_K (\Shi_c \mu, \Shi_c \nu) = d_K (\mu, \nu)$ and $d_L (\Shi_c \mu, \Shi_c \nu) = d_L (\mu, \nu)$ for $\mu, \nu \in \calM(\R)$.
In other words, the shift operation preserves these distances.
Similarly, for $c \ne 0$, $\cdf_{\Dil_c \mu}(x) = \cdf_\mu (x / c)$, and hence
\[
d_K (\Dil_c \mu, \Dil_c \nu) = d_K (\mu, \nu).
\]
Furthermore,
\begin{equation} \label{eq:reflected_cdf_relation}
\cdf_{\mu}(x) + \cdf_{\widehat{\mu}}(-x) = 1 + \mu(\{x\}),
\end{equation}
for all $x \in \R$.

\begin{definition}[A partial order on measures] \label{def:PartialOrderOnMeasures}
Let $\mu, \nu \in \calM (\R)$.
If $\cdf_\mu (x) \le \cdf_\nu (x)$ for all $x \in \R$, then we write $\nu \le \mu$.
\end{definition}
It follows directly from the definition that $\mu \le \Shi_c \mu$ for every $c > 0$.

\begin{definition}[Cut-up and cut-down measures] \label{def:cut_measures}
Given a measure $\mu \in \calM (\R)$ with CDF $\cdf_\mu$, we define the cut-up measure at $a \in \R$ as the measure $\mu|^a \in \calM ((-\infty, a])$ with CDF
\[
\cdf_{\mu|^a}(x) =
\begin{cases}
1 & \quad \text{if } x\ge a, \\
\cdf_\mu(x) & \quad \text{if } x< a.
\end{cases}
\]
Similarly, we define the cut-down measure at $a \in \R$ as the measure $\mu|_a \in \calM ([a, \infty))$ with CDF
\[
\cdf_{\mu|_a}(x) =
\begin{cases}
0 & \quad \text{if } x< a, \\
\cdf_\mu(x) & \quad \text{if } x\ge a.
\end{cases}
\]
For $a > 0$, we use the notation
\[
(\mu)_a := (\mu|^a)|_{-a} \in \calM ([-a, a]).
\]
\end{definition}

The following basic properties of cut-down and cut-up measures follow directly from their definitions:
\begin{itemize}
    \item $\mu|_a \weakto \mu$ as $a \to - \infty$, and $\mu|^a \weakto \mu$ as $a \to \infty$.
    \item Since $\cdf_{\mu|_a}(x) \le \cdf_{\mu}(x) \le \cdf_{\mu|^a}(x)$ for all $x \in \R$, we have $\mu|^a \le \mu \le \mu|_a$.
\end{itemize}

\subsection{Free probability} \label{sec:FreeProbability}

We briefly review the regular properties of free additive and multiplicative convolutions, which are our main research topics.
For more basic knowledge, the readers are referred to the standard textbooks in free probability \cite{NS, MS, VDN92}.

Given measures $\mu, \nu \in \calM(\R)$, the free additive convolution $\mu \boxplus \nu$ is defined as the distribution of $X + Y$, where $X$ and $Y$ are freely independent non-commutative random variables distributed as $\mu$ and $\nu$, respectively.
Additionally, if $\nu \ge 0$ (equivalently $Y \ge 0$), the free multiplicative convolution $\mu \boxtimes \nu$ is defined as the distribution of $\sqrt{Y}X\sqrt{Y}$.
Bercovici and Voiculescu proved the following monotonicity for the free convolutions \cite[Propositions 4.15 and 4.16]{BV93}:

\begin{proposition} \label{prop:monotonicity_free_convolutions}
Let $\mu_1, \mu_2$, and $\nu$ be probability measures on $\R$ such that $\mu_1 \le \mu_2$. Then $\mu_1 \boxplus \nu \le \mu_2 \boxplus \nu$.
If $\nu \in \calM(\R_{\ge 0})$ then $\mu_1 \boxtimes \nu \le \mu_2 \boxtimes \nu$.
\end{proposition}
Using these inequalities, we can easily obtain the following basic fact.

\begin{corollary} \label{cor:support_bounds_free_convolutions}
For $a,b \in \R$, if $\mu \in \calM((-\infty, a])$ and $\nu \in \calM((-\infty, b])$ then $\mu \boxplus \nu \in \calM((-\infty, a+b])$.
Similarly, for $a,b \in \R_{\ge 0}$, if $\mu \in \calM([0, a])$ and $\nu \in \calM([0,b])$ then $\mu \boxtimes \nu \in \calM([0, ab])$.
\end{corollary}
\begin{proof}
By assumption, $\mu \le \delta_a$ and $\nu \le \delta_b$. Hence, by Proposition \ref{prop:monotonicity_free_convolutions}, we have
\[
\mu \boxplus \nu \le \mu \boxplus \delta_b \le \delta_a \boxplus \delta_b = \delta_{a+b}.
\]
A similar argument applies to the free multiplicative convolution $\boxtimes$.
\end{proof}

Bercovici and Voiculescu also showed the monotonicity for the Kolmogorov and Lévy distances with respect to free additive convolution:
\begin{proposition}[{\cite[Proposition 4.13]{BV93}}] \label{prop:monotonicity_distance_additive}
If $\mu_1, \mu_2, \nu_1$, and $\nu_2$ are probability measures on $\R$, then
\begin{equation} \label{eq:monotonicity_levy_additive}
d_L(\mu_1 \boxplus \nu_1, \mu_2 \boxplus \nu_2) \le d_L(\mu_1, \mu_2) + d_L(\nu_1, \nu_2),
\end{equation}
and
\[
d_K(\mu_1 \boxplus \nu_1, \mu_2 \boxplus \nu_2) \le d_K(\mu_1, \mu_2) + d_K(\nu_1, \nu_2).
\]
\end{proposition}
Notice that inequality \eqref{eq:monotonicity_levy_additive}, by the triangle inequality for the Lévy distance, is equivalent to
\[
d_L(\mu \boxplus \rho, \nu \boxplus \rho) \le d_L(\mu, \nu)
\]
for $\mu, \nu, \rho \in \calM(\R)$; the same holds for the Kolmogorov distance.

The analogue also holds for free multiplicative convolution, but only with respect to the Kolmogorov distance.
\begin{proposition}[{\cite[Proposition 4.14]{BV93}}] \label{prop:monotonicity_distance_multiplicative_kolmogorov}
If $\mu_1, \mu_2, \nu_1$, and $\nu_2$ are probability measures on $\R_{\ge 0}$, then
\[
d_K(\mu_1 \boxtimes \nu_1, \mu_2 \boxtimes \nu_2) \le d_K(\mu_1, \mu_2) + d_K(\nu_1, \nu_2).
\]
\end{proposition}
Note that a similar inequality with respect to the Lévy distance does not hold; for instance, consider $\mu_1 = \delta_0$, $\mu_2 = \delta_{1/2}$, and $\nu_1 = \nu_2 = \delta_2$.
Actually, in \cite{BV93}, the inequality
\begin{equation} \label{eq:monotonicity_multiplicative_one_sided}
d_K(\mu \boxtimes \rho, \nu \boxtimes \rho) \le d_K(\mu, \nu)
\end{equation}
was proved for $\mu, \nu \in \calM(\R)$ and $\rho \in \calM(\R_{\ge 0})$. We use this in Section \ref{sec:convergence}.

Bercovici and Voiculescu also pointed out that the following atom regularity holds.
\begin{proposition}[{\cite[Theorem 7.4]{BV98}}] \label{prop:atom_regularity_additive}
Let $\mu, \nu \in \calM(\R)$ and let $\gamma$ be a real number.
The following are equivalent:
\begin{enumerate}
    \item[(i)] $\gamma$ is an atom of $\mu \boxplus \nu$;
    \item[(ii)] there exist atoms $\alpha$ of $\mu$ and $\beta$ of $\nu$ such that $\gamma = \alpha + \beta$ and $\mu(\{\alpha\}) + \nu(\{\beta\}) >1$.
    In this case, we have $(\mu \boxplus \nu)(\{\gamma\}) = \mu(\{\alpha\}) + \nu(\{\beta\})-1$.
\end{enumerate}
\end{proposition}

Furthermore, we can prove the following.
\begin{proposition} \label{prop:cdf_at_atom_additive}
Given $\mu, \nu \in \calM(\R)$, if $\alpha, \beta \in \R$ are atoms of $\mu, \nu$, respectively, such that $\mu(\{\alpha\}) + \nu(\{\beta\}) >1$, then
\[
\cdf_{\mu \boxplus \nu}(\alpha+\beta) = \cdf_{\mu}(\alpha) + \cdf_{\nu}(\beta) - 1.
\]
\end{proposition}
\begin{proof}
Define the cut-down measures $\mu|_\alpha$ and $\nu|_\beta$. We have $\mu \le \mu|_\alpha$ and $\nu \le \nu|_\beta$.
Then, by Proposition \ref{prop:monotonicity_free_convolutions}, $\mu \boxplus \nu \le \mu|_\alpha \boxplus \nu|_\beta$, which implies
\[
\cdf_{\mu \boxplus \nu} (\alpha+\beta) \ge \cdf_{\mu|_\alpha \boxplus \nu|_\beta}(\alpha+\beta).
\]
Note that $\mu|_\alpha \boxplus \nu|_\beta \in \calM([\alpha+\beta, \infty))$ by Corollary \ref{cor:support_bounds_free_convolutions}.
The mass at the point $\alpha+\beta$ is given by Proposition \ref{prop:atom_regularity_additive} as
$\mu|_\alpha \boxplus \nu|_\beta(\{\alpha+\beta\}) = \mu|_\alpha (\{\alpha\}) + \nu|_\beta (\{\beta\}) - 1$.
Since $\mu|_\alpha(\{\alpha\}) = \mathcal{F}_\mu(\alpha)$ and $\nu|_\beta(\{\beta\}) = \mathcal{F}_\nu(\beta)$, and $\alpha+\beta$ is the leftmost point in the support of $\mu|_\alpha \boxplus \nu|_\beta$, we have $\cdf_{\mu|_\alpha \boxplus \nu|_\beta}(\alpha+\beta) = \mu|_\alpha \boxplus \nu|_\beta(\{\alpha+\beta\})$.
Thus,
\begin{equation} \label{eq:cdf_inequality_additive_1}
    \cdf_{\mu \boxplus \nu} (\alpha+\beta) \ge \cdf_{\mu}(\alpha) + \cdf_{\nu}(\beta) - 1.
\end{equation}
The same argument applies to the reflected measures $\widehat{\mu}$ and $\widehat{\nu}$, noting that $\widehat{\mu \boxplus \nu} = \widehat{\mu} \boxplus \widehat{\nu}$, with atoms $(-\alpha, -\beta, -\alpha -\beta)$. This yields
\begin{equation} \label{eq:cdf_inequality_additive_2}
    \cdf_{\widehat{\mu \boxplus \nu}} (- \alpha -\beta) \ge \cdf_{\widehat{\mu}}(-\alpha) + \cdf_{\widehat{\nu}}(-\beta) - 1.
\end{equation}
Combining inequalities \eqref{eq:cdf_inequality_additive_1}, \eqref{eq:cdf_inequality_additive_2} and Equation \eqref{eq:reflected_cdf_relation}, we obtain
\begin{equation*}
    1+ \mu\boxplus\nu (\{\alpha+\beta\}) \ge \mu(\{\alpha\}) + \nu(\{\beta\}).
\end{equation*}
Since equality holds by Proposition \ref{prop:atom_regularity_additive}, it implies that equality must hold in both \eqref{eq:cdf_inequality_additive_1} and \eqref{eq:cdf_inequality_additive_2}.
\end{proof}

The corresponding result for $\boxtimes$ was first established for $\mu, \nu \in \calM(\R_{\ge 0})$ by Belinschi and finally extended for $\mu \in \calM(\R)$ by Arizmendi et al.
\begin{proposition}[{\cite[Theorem 4.1]{Bel03} and \cite[Corollary 6.6]{ACSY}}] \label{prop:atom_regularity_multiplicative}
Suppose that $\mu \in \calM(\R)$ and $\nu \in \calM(\R_{\ge 0})$.
\begin{enumerate}
    \item[(i)] $(\mu \boxtimes \nu) (\{0\}) = \max\{\mu(\{0\}), \nu (\{0\})\}$.
    \item[(ii)] Let $\gamma \neq 0$. We have $(\mu \boxtimes \nu) (\{\gamma\})>0$ if and only if there exist $\alpha, \beta \in \R$ such that $\gamma = \alpha \beta$ and $\mu(\{\alpha\}) + \nu(\{\beta\})>1$.
    In this case, we have $(\mu \boxtimes \nu) (\{\gamma\}) = \mu(\{\alpha\}) + \nu(\{\beta\})-1$.
\end{enumerate}
\end{proposition}

In particular, if $\mu, \nu \in \calM(\R_{\ge 0})$, we can obtain the corresponding result to Proposition \ref{prop:cdf_at_atom_additive}, whose proof is almost identical.

\begin{proposition}\label{prop:cdf_at_atom_multiplicative}
Given $\mu, \nu \in \calM(\R_{\ge 0})$ with $\alpha, \beta \in \R_{>0}$ being atoms of $\mu, \nu$, respectively, such that $\mu(\{\alpha\}) + \nu(\{\beta\}) >1$, we have
\[
\cdf_{\mu \boxtimes \nu}(\alpha\beta) = \cdf_{\mu}(\alpha) + \cdf_{\nu}(\beta) - 1.
\]
\end{proposition}
\begin{proof}
In the same manner as the proof of Proposition \ref{prop:cdf_at_atom_additive}, using cut-down measures $\mu|_\alpha$ and $\nu|_\beta$, we obtain
\[
\cdf_{\mu \boxtimes \nu}(\alpha\beta) \ge \cdf_\mu(\alpha) + \cdf_\nu(\beta) -1,
\]
because $\mu \boxtimes \nu \le \mu|_\alpha \boxtimes \nu|_\beta$  implies $\cdf_{\mu \boxtimes \nu}(\alpha\beta) \ge \cdf_{\mu|_\alpha \boxtimes \nu|_\beta}(\alpha\beta) = \cdf_\mu(\alpha) + \cdf_\nu(\beta) -1$.

Similarly, using cut-up measures, $\mu|^\alpha \le \mu$ and $\nu|^\beta \le \nu$. Thus $\mu|^\alpha \boxtimes \nu|^\beta \le \mu \boxtimes \nu$. This implies $\widehat{\mu \boxtimes \nu} \le \widehat{\mu|^\alpha \boxtimes \nu|^\beta}$.
Therefore, $\cdf_{\widehat{\mu \boxtimes \nu}} (-\alpha \beta) \ge \cdf_{\widehat{\mu|^\alpha \boxtimes \nu|^\beta}} (-\alpha \beta)$.
Using Equation \eqref{eq:reflected_cdf_relation}, Corollary \ref{cor:support_bounds_free_convolutions}, and Proposition \ref{prop:atom_regularity_multiplicative}, we have:
\begin{align*}
\cdf_{\widehat{\mu|^\alpha \boxtimes \nu|^\beta}} (-\alpha \beta)
&= 1 + \mu|^\alpha \boxtimes \nu|^\beta(\{\alpha\beta\}) - \cdf_{\mu|^\alpha \boxtimes \nu|^\beta}(\alpha\beta) \\
&= \mu|^\alpha(\{\alpha\}) + \nu|^\beta(\{\beta\}) -1,
\end{align*}
since $\cdf_{\mu|^\alpha \boxtimes \nu|^\beta}(\alpha\beta) = 1$ as the support of $\mu|^\alpha \boxtimes \nu|^\beta$ is in $(-\infty, \alpha\beta]$ (for $\alpha, \beta > 0$).
We also have $\mu|^\alpha(\{\alpha\}) = \mu([\alpha, \infty))$ and $\nu|^\beta(\{\beta\}) = \nu([\beta, \infty))$.
So, $\cdf_{\widehat{\mu \boxtimes \nu}}(-\alpha\beta) \ge \mu([\alpha, \infty)) + \nu([\beta, \infty)) - 1$.
The rest of the argument, showing that equality holds, is similar to the discussion in the proof of Proposition \ref{prop:cdf_at_atom_additive}.
\end{proof}

\subsection{Polynomials} \label{sec:polynomials}

Let the symbol $\C_d[x]$ denote the set of polynomials over $\C$ of degree at most $d$, i.e., $\C_d[x] = \{ \sum_{k=0}^d a_k x^k \mid a_k \in \C\}$.
We use $\calP_d(\C) \subset \C_d[x]$ to denote the set of monic polynomials of degree $d$.
For a subset $K$ of the complex plane, $\calP_d (K)$ denotes the set of $p \in \calP_d(\C)$ such that the roots of $p$ are contained in $K$.
For instance, $\calP_d (\R)$ is the set of monic polynomials of degree $d$ having only real roots.
Given $p (x) = \prod_{i=1}^d (x-\lambda_i) \in \calP_d(\C)$, the symbol $\coef_k^{(d)}(p)$ denotes the normalized $k$-th elementary symmetric polynomial on the roots of $p$, namely
\[
\coef_k^{(d)}(p) = \binom{d}{k}^{-1} \sum_{1 \le i_1 < \dots < i_k \le d} \lambda_{i_1} \dots \lambda_{i_k},
\]
with the convention that $\coef_0^{(d)}(p) = 1$.
Hence, we can express any $p \in \calP_d(\C)$ as
\begin{equation} \label{eq:polynomial.coefficients}
p (x) = \prod_{i=1}^d (x-\lambda_i) = \sum_{k=0}^d (-1)^k \binom{d}{k} \coef_k^{(d)} (p) x^{d-k}.
\end{equation}
The empirical root distribution of $p \in \calP_d(\C)$ is defined as
\[
\meas{p} := \frac{1}{d} \sum_{i=1}^d \delta_{\lambda_i}.
\]
We always identify a polynomial $p$, its coefficients $\{\coef_k^{(d)} (p)\}_{k=0}^d$, and its roots $\{\lambda_i\}_{i=1}^d$, or equivalently, its empirical root distribution $\meas{p}$.
When the coefficient $\coef_k^{(d)} (p)$ is represented, its superscript $(d)$ and the polynomial $(p)$ will often be omitted if $d$ and $p$ are clear from the context.
Note that in most situations, $\coef_0^{(d)} (p) = 1$, since we usually restrict our attention to monic polynomials.

The map $\meas{\cdot}$ provides a bijection from $\calP_d(\C)$ onto the set
\[
\calM_d(\C) := \left\{ \frac{1}{d} \sum_{i=1}^d \delta_{\lambda_i} \mid \lambda_1, \dots, \lambda_d \in \C \right\}.
\]
Moreover, we have that $\meas{\calP_d (K)} = \calM_d (K)$ for every $K \subset \C$, where the latter is defined as $\calM_d (K) = \calM_d (\C) \cap \calM (K)$.
Notice that for every $d$, the subset $\calM_d(\C)$ is invariant under the transformations in Definition \ref{def:transformation.on.measures}; thus, we can use the bijection $\meas{\cdot}$ to define analogous transformations on the set of polynomials.

\begin{definition}[Transformations on polynomials] \label{def:polynomial_transformations}
Let $p \in \calP_d(\C)$.
\begin{enumerate}
    \item For $c \in \C$, we define the \textit{shifted polynomial} as $[\Shi_c (p)](x) := p(x-c)$.
    \item For $c\neq 0$, we define the \textit{dilated polynomial} as $[\Dil_c (p)](x) := c^d p(x/c)$.
    If $c=0$, we use the convention $[\Dil_0 (p)](x) = x^d$.
    \item The \textit{reflected polynomial} $\widehat{p} \in \calP_d(\C)$ is given by $\widehat{p}(x) := [\Dil_{-1} (p)](x) = (-1)^d p(-x)$.
    \item If $p \in \calP_d(\C\setminus\{0\})$ (i.e., $p(0) \ne 0$, so $\coef_d(p) \ne 0$), the \textit{reversed polynomial} $p^{\langle -1 \rangle} (x)$ is defined as $x^d p(1/x)/\coef_d(p)$.
\end{enumerate}
\end{definition}
A partial order on $\calP_d (\R)$ is also induced from $\calM_d (\R)$ (see Definition \ref{def:PartialOrderOnMeasures}):
\[
p \le q \quad  \overset{\text{def}}{\Longleftrightarrow} \quad \meas{p} \le \meas{q}
\]
for $p,q \in \calP_d (\R)$.

We can also define the corresponding cut-up and cut-down operations on polynomials using the bijection $\meas{\cdot}$ between $\calP_d(\R)$ and $\calM_d(\R)$.
That is, given $p \in \calP_d(\R)$ with roots $\lambda_i(p)$ for $i=1,\dots,d$, and $a\in\R$,
the \textit{cut-up polynomial} $p|^a \in \calP_d(\R_{\le a})$ has roots $\lambda_i (p|^a)$ defined by
\[
\lambda_i (p|^a) = \min(\lambda_i (p), a),
\]
and the \textit{cut-down polynomial} $p|_a \in \calP_d(\R_{\ge a})$ has roots $\lambda_i (p|_a)$ defined by
\[
\lambda_i (p|_a) = \max(\lambda_i (p), a).
\]
Using the previous definitions:
\begin{align*}
\lambda_i (p|^a) &=
\begin{cases}
\lambda_i (p) & \quad \text{if } \lambda_i (p) < a, \\
a & \quad \text{if } \lambda_i (p) \ge a;
\end{cases}\\
\lambda_i (p|_a) &=
\begin{cases}
a & \quad \text{if } \lambda_i (p) \le a, \\
\lambda_i (p) & \quad \text{if } \lambda_i (p) > a.
\end{cases}
\end{align*}
If $(p_d)_{d \in \N}$ is a sequence of polynomials such that $\meas{p_d} \weakto \mu$, then $\meas{p_d|_a} \weakto \mu|_a$ and $\meas{p_d|^a} \weakto \mu|^a$.
Similarly, $(p)_a := (p|^a)|_{-a} \in \calP_d ([-a, a])$.

Let the symbol $D$ denote the derivation operator $d/dx$ for polynomials.
For $j \le d$, the derivative map $\partial_{j|d}: \calP_d(\C) \to \calP_j(\C)$ is given by
\[
\partial_{j|d} p := \frac{D^{(d-j)}p}{(d)_j},
\]
where $(d)_j = d(d-1)\dots(d-j+1) = d!/(d-j)!$ is the normalization constant.

\subsection{Finite free probability} \label{sec:finitefree}
This section briefly reviews the definitions and essential properties of finite free probability, highlighting its crucial connections and analogies to the established results in free probability, which will serve as a foundation for the subsequent discussion.

\begin{definition}[Finite Free Convolutions] \label{def:finite-free-convolution}
Let $p, q \in \C_d[x]$ be polynomials of degree at most $d$.
\begin{enumerate}
\item The finite free additive convolution $p \boxplus_d q$ is defined by
\[
\coef_k^{(d)}(p \boxplus_d q) = \sum_{i=0}^k \binom{k}{i} \coef_{i}^{(d)}(p) \coef_{k-i}^{(d)}(q),
\]
for $0 \le k \le d$.
\item The finite free multiplicative convolution $p \boxtimes_d q$ is defined by
\[
\coef_k^{(d)}(p \boxtimes_d q) = \coef_{k}^{(d)}(p) \coef_{k}^{(d)}(q),
\]
for $0 \le k \le d$.
\end{enumerate}
\end{definition}

By definition, it immediately follows that these operations are bilinear, commutative, and associative.
Furthermore, the shift, dilation, and derivatives can be represented using finite free convolutions.

\begin{example} Let $p \in \calP_d(\C)$.
\begin{enumerate}
    \item If $q(x) = (x-c)^d$ for $c \in \C$, then $p \boxplus_d q (x) = p(x-c) = [\Shi_c(p)](x)$ and also $p \boxtimes_d q = \Dil_c (p)$.
    \item If $q^{(k)}(x) = x^{d-k}(x-1)^{k}$ for $k=0, \dots, d$, then $p \boxtimes_d q^{(k)} (x) = x^{d-k} \diff_{k|d}p (x)$.
    \item For $q\in\calP_d(\C)$, we have $\widehat{p \boxplus_d q} = \widehat{p} \boxplus_d \widehat{q}$ and $\widehat{p \boxtimes_d q} = \widehat{p} \boxtimes_d q = p \boxtimes_d \widehat{q}$.
    \item Assume $p, q \in \calP_d(\C \setminus \{0\})$.
    Then, ${(p \boxtimes_d q)}^{\langle -1 \rangle} = p^{\langle -1 \rangle} \boxtimes_d q^{\langle -1 \rangle}$ for $p,q \in \calP_d(\C \setminus \{0\})$.
\end{enumerate}
\end{example}

Moreover, the finite free convolutions preserve the real-rootedness of polynomials under many situations.

\begin{proposition} \label{prop:preserve}
Let $d$ be a fixed degree.
\begin{enumerate}
\item If $p, q \in \calP_d (\R)$, then $p \boxplus_d q \in \calP_d (\R)$.
\item If $p \in \calP_d (\R)$ and $q \in \calP_d (\R_{\ge 0})$, then $p \boxtimes_d q \in \calP_d (\R)$.
\item If $p, q \in \calP_d (\R_{\ge 0})$, then $p \boxtimes_d q \in \calP_d (\R_{\ge 0})$.
\end{enumerate}
\end{proposition}

\begin{definition}[Interlacing] \label{def:interlacing}
Let $q \in \calP_d (\R)$.
We say that $p \in \calP_d(\R)$ interlaces $q$, denoted by $p \lessdot q$, if their roots $\lambda_i(p)$ and $\lambda_i(q)$, when ordered non-decreasingly, satisfy
\[
\lambda_1 (p) \le \lambda_1 (q) \le \lambda_{2} (p) \le \lambda_{2} (q) \le \dots \le \lambda_d (p) \le \lambda_d (q).
\]
Similarly, we say that $p \in \calP_{d-1}(\R)$ interlaces $q \in \calP_d(\R)$, denoted by $p \lessdot q$, if
\[
\lambda_1 (q) \le \lambda_{1} (p) \le \lambda_{2} (q) \le \dots \le \lambda_{d-1} (p) \le \lambda_d (q).
\]
\end{definition}
For $p, q \in \calP_d(\R)$, it is clear that if $p \lessdot q$, then $p \le q$ (in the sense of Definition \ref{def:PartialOrderOnMeasures}).
Besides, it is well known that $p,q \in \calP_d(\R)$ having $p \lessdot q$ is related to the real-rootedness of linear combinations $ap+bq$.

\begin{proposition} \label{prop:InterlacingPreserving}
If $p, q \in \calP_d (\R)$ and $p \lessdot q$, then
\begin{align*}
r \in \calP_d (\R) &\longimplies p \boxplus_d r \lessdot q \boxplus_d r, \\
r \in \calP_d (\R_{\ge 0}) &\longimplies p \boxtimes_d r \lessdot q \boxtimes_d r.
\end{align*}
\end{proposition}

For $p \le q$, there exists a family of polynomials $(p_t)_{t \in [0,1]}$ with $p_0 = p$, $p_1 = q$, such that $p_s \lessdot p_t$ for $s \le t$, effectively interpolating $p$ and $q$ through interlacing polynomials.
Hence, we obtain the following result, which is a finite free analogue of Proposition \ref{prop:monotonicity_free_convolutions}.

\begin{proposition}[{\cite[Theorem 4.4]{AFPU}}] \label{prop:OrderPreserving}
If $p, q \in \calP_d (\R)$ and $p \le q$, then
\begin{align*}
r \in \calP_d (\R) &\longimplies p \boxplus_d r \le q \boxplus_d r; \\
r \in \calP_d (\R_{\ge 0}) &\longimplies p \boxtimes_d r \le q \boxtimes_d r.
\end{align*}
\end{proposition}

Identifying a polynomial $p \in \calP_d (\R)$ with its root distribution $\meas{p}$, we write
\[
\cdf_p (x) := \cdf_{\meas{p}} (x) \qquad \text{ for } x\in\R.
\]
Similarly, for $p, q \in \calP_d (\R)$, we write $d_K (p,q) = d_K (\meas{p}, \meas{q})$ and $d_L (p,q) = d_L (\meas{p}, \meas{q})$.

Using the above monotonicity of finite free convolutions, we can provide an alternative proof of the bounds on the roots of convoluted polynomials in \cite[Theorems 1.3 and 1.6]{MSS22} in the same as the proof of Corollary \ref{cor:support_bounds_free_convolutions}.
In the following, for $p \in \calP_d(\R)$, let $\lambda_{\mathrm{max}}(p)$  denote the maximum roots of $p$.

\begin{corollary} [{\cite[Theorems 1.3 and 1.6]{MSS22}}] \label{cor:minRootsFiniteFreeMultiplicative}
Given $p, q \in \calP_d(\R)$, we have $\lambda_{\mathrm{max}}(p\boxplus_d q) \le \lambda_{\mathrm{max}}(p) + \lambda_{\mathrm{max}}(q)$.
If $p, q \in \calP_d(\R_{\ge 0})$, then $\lambda_{\mathrm{max}}(p\boxtimes_d q) \le \lambda_{\mathrm{max}}(p)\lambda_{\mathrm{max}}(q)$.
\end{corollary}

Arizmendi and Perales construed the theory of finite free cumulants \cite{AP} inspired by the work of finite $R$-transform by Marcus \cite{Marcus}, and they obtained the convergence of finite free additive convolution $\boxplus_d$ to free additive convolution $\boxplus$ under the assumption that the limiting measures have compact support.
\begin{proposition}[{\cite[Corollary 5.5]{AP}}] \label{prop:convergence-finite-free-additive}
Let $\mu$ and $\nu$ be real probability measures with compact support.
For $d\in\N$, let $p_d$ and $q_d$ be monic polynomials of degree $d$.
If, as $d\to\infty$, $\meas{p_d} \weakto \mu$ and $\meas{q_d} \weakto \nu$, then we have the weak convergence $\meas{p_d \boxplus_d q_d} \weakto \mu \boxplus \nu$.
\end{proposition}
In the next work \cite{AGVP}, Arizmendi, Perales, and Garza-Vargas obtained the multiplicative formula for finite free cumulants and then proved the convergence of $\boxtimes_d$ to $\boxtimes$ under a similar condition.
\begin{proposition}[{\cite[Theorem 1.4]{AGVP}}]\label{prop:convergence-finite-free-multiplicative}
Let $\mu$ and $\nu$ be probability measures supported on a compact subset of the real line.
Let $(p_d)_{d=1}^\infty$ and $(q_d)_{d=1}^\infty$ be sequences of real-rooted polynomials, and assume that the $q_d$ have only non-negative roots.
If the empirical root distributions of these sequences of polynomials converge weakly to $\mu$ and $\nu$ respectively, then the empirical root distributions of the sequence $(p_d \boxtimes_d q_d)_{d=1}^\infty$ converge weakly to $\mu \boxtimes \nu$.
\end{proposition}

Recently, Arizmendi et al. \cite{AFPU} constructed the finite $S$-transform for $\calP_d(\R_{\ge 0})$.
As a result, for the case where both $p_d, q_d \in \calP_d(\R_{\ge 0})$, the condition on compact supports can be dropped for the multiplicative convolution.

\begin{proposition}[{\cite[Proposition 10.1]{AFPU}}] \label{prop:convergenceFiniteFreeMultiplicative}
Let $p_d, q_d \in \calP_d (\R_{\ge 0})$ and assume their root distributions $\meas{p_d}$ and $\meas{q_d}$ converge weakly to $\mu, \nu \in \calM (\R_{\ge 0})$ as $d\to\infty$, respectively.
Then, the empirical root distributions $\meas{p_d \boxtimes_d q_d}$ converge weakly to $\mu \boxtimes \nu$ as $d\to\infty$.
\end{proposition}

More recently, Jalowy, Kabluchko, and Marynych \cite{JKM} dealt with these kinds of problems in \cite[Theorems 3.5 and 3.7]{JKM} by using exponential profiles; however, they also required similar constraints on the support.

\section{Atoms for finite free convolutions} \label{sec:atoms}

This section investigates the nature of atoms in finite free convolutions, providing detailed proofs for the conditions governing their existence and multiplicities.
We begin in Section \ref{sec:atoms-multiplicative} with the multiplicative convolution $\boxtimes_d$, building upon the foundational work of Kostov and Shapiro on the Schur--Szeg\H{o} composition \cite{KS}.
While reviewing their key insights, we also present alternative derivations for several results using the framework of finite free probability (with Proposition \ref{prop:simple-finite-free-multiplicative} being an exception where their original arguments are noted).
Subsequently, Section \ref{sec:atoms-additive} establishes analogous theorems for the additive convolution $\boxplus_d$, largely by adapting the methodologies employed for the multiplicative case.

\subsection{Finite free multiplicative convolution}
\label{sec:atoms-multiplicative}
Interestingly, Kostov and Shapiro already studied the atoms of the Schur-Szeg\H{o} composition in 2006 \cite{KS}.
In this subsection, we review their work from the viewpoint of finite free probability.
Given polynomials of degree $d$,
\[  p (x) = \sum_{k=0}^d \binom{d}{k} a_k x^k \quad \text{and} \quad q(x) = \sum_{k=0}^d \binom{d}{k} b_k x^k,\]
the Schur-Szeg\H{o} composition is given by
\[  p * q (x) = \sum_{k=0}^d \binom{d}{k} a_k b_k x^k.\]
They were interested in the roots of the Schur-Szeg\H{o} composition $p * q$ and obtained helpful results for our goal concerning atom information for finite free convolutions.
The connection to finite free multiplicative convolution is clear as follows: $\widehat{p \boxtimes_d q} = p * q$ under appropriate identification of coefficients.
Thus, we can interpret their results in terms of finite free multiplicative convolutions.
For the reader's convenience, we restate Kostov and Shapiro's relevant results in our terminology.

\begin{proposition}[Theorem 1.4 in \cite{KS}] \label{prop:boxtimes-atom}
Given two polynomials $p$ and $q$ in $\calP_d (\C)$ such that $\alpha \ne 0$ and $\beta \ne 0$ are roots respectively of $p$ and $q$ of multiplicity $m_\alpha^p$ and $m_\beta^q$ with $m = m_\alpha^p + m_\beta^q - d \ge 0$, one has that $\alpha \beta$ is a root of $p \boxtimes_d q$ of multiplicity $m$.
(If $m=0$, $\alpha \beta$ is not a root of $p \boxtimes_d q$.)
\end{proposition}

In the proposition's setting, if $m>0$, we call $\alpha \beta$ a \textit{trivial root} of $p \boxtimes_d q$, and the triplet $(\alpha, \beta, \alpha \beta)$ is termed an \textit{atom triplet} of $(p, q, p \boxtimes_d q)$. This corresponds to what Kostov and Shapiro called ``A-roots'' \cite{KS}.
Also, if $p$ and $q$ has a root at the origin with multiplicity $m^p_0$ and $m^q_0$, then $p \boxtimes_d q$ also has a root at the origin with multiplicity $\max(m^p_0, m^q_0)$ (since $\coef_d(p \boxtimes_d q) = \coef_d(p)\coef_d(q)$, etc.); such roots at the origin are also considered trivial.
The other roots of $p \boxtimes_d q$ are said to be \textit{non-trivial} (``B-roots'' in \cite{KS}).
The proof of Proposition \ref{prop:boxtimes-atom} in \cite{KS} involved direct computation with coefficients, similar to the proof of Proposition \ref{prop:boxplus_atom} in this paper.
However, we present an alternative proof here for completeness and to showcase the utility of finite free probability techniques.

For this alternative proof, we review an expression for $\boxtimes_d$ using a differential operator approach, cf. \cite{MSS22, Mirabelli}.
Let $D = d/dx$. Define $r_d^{(0)} (x) = (x-1)^d$, $r_d^{(1)} (x) = x(x-1)^{d-1}$, and $r_d^{(k)} = (r_d^{(1)})^{\boxtimes_d k} = (x(x-1)^{d-1})^{\boxtimes_d k}$ for $k \ge 2$.
For a fixed degree $d$, these polynomials $\{r_d^{(k)}\}_{k=0}^\infty$ are denoted simply by $\{r^{(k)}\}_{k=0}^\infty$.
Clearly, $r^{(k)} \boxtimes_d r^{(l)} = r^{(k+l)}$ for any $k, l \ge 0$.

\begin{lemma} \label{lem:uniqueExpression}
The set of polynomials $\{r^{(k)}\}_{k=0}^d$ forms a basis for the $\C$-vector space $\C_d [x]$.
\end{lemma}

\begin{proof}
Consider the operator $\mathcal{D}_x = \frac{x D}{d}$. The action of $\boxtimes_d$ with $r^{(1)}(x)$ on $p \in \calP_d(\C)$ is given by $r^{(1)} \boxtimes_d p (x) = \frac{x Dp(x)}{d} = \mathcal{D}_x p(x)$.
Thus, $r^{(k)}(x) = (\mathcal{D}_x)^k (x-1)^d$ for $k \ge 0$.
The polynomial $r^{(k)}(x)$ has a root at $x=1$ of multiplicity $d-k$ for $0 \le k \le d$.
Consider a linear combination $\sum_{k=0}^d a_k r^{(k)} (x) = 0$.
Evaluating at $x=1$: $a_d r^{(d)}(1) = 0$, we obtain $a_d = 0$.
Now consider $\sum_{k=0}^{d-1} a_k r^{(k)}(x) = 0$.
Similarly, differentiating and evaluating at $x=1$ imply $a_{d-1}=0$.
This inductive step shows all $a_k =0$, and hence $\{r^{(k)}\}_{k=0}^d$ is linearly independent and is a basis of $\C_d [x]$.
\end{proof}

By Lemma \ref{lem:uniqueExpression}, any polynomial $p \in \calP_d(\C)$ can be uniquely expressed as $p(x) = \sum_{k=0}^d a_k r^{(k)}(x)$.
This allows us to associate $p(x)$ with a polynomial $P(X) = \sum_{k=0}^d a_k X^k \in \C_d[X]$ such that $p(x) = P(\mathcal{D}_x)(x-1)^d$.
For $p, q \in \calP_d(\C)$ with associated $P(X)=\sum a_k X^k$ and $Q(X)=\sum b_k X^k$, one has
\begin{align*}
p \boxtimes_d q (x) &= \left(\sum_{i=0}^d a_i r^{(i)}(x)\right) \boxtimes_d \left(\sum_{j=0}^d b_j r^{(j)}(x)\right) = \sum_{i=0}^d \sum_{j=0}^d a_i b_j (r^{(i)} \boxtimes_d r^{(j)})(x) \\
&= \sum_{i=0}^d \sum_{j=0}^d a_i b_j r^{(i+j)}(x) = \sum_{m=0}^{2d} \left(\sum_{i+j = m} a_i b_{j} \right) r^{(m)} (x) \\
&= \sum_{m=0}^{2d} \left(\sum_{i+j=m} a_i b_{j} \right) (\mathcal{D}_x)^m (x-1)^d = [(P Q)(\mathcal{D}_x)] (x-1)^d,
\end{align*}
where $(PQ)(X) = P(X)Q(X)$.

\begin{remark} \label{rem:uniqueExpression}
The product $P(X)Q(X)$ may have degree up to $2d$, so $(PQ)(\mathcal{D}_x)(x-1)^d$ is not necessarily in the span of $\{r^{(k)}\}_{k=0}^d$ without further reduction if $k > d$.
To express $p \boxtimes_d q \in \calP_d(\C)$ in the basis $\{r^{(k)}\}_{k=0}^d$, one must use relations to reduce terms $r^{(k)}$ for $k>d$ such as
\[
r^{(d+1)} = \sum_{k=0}^d c^{(k)} r^{(k)},
\]
where $\{c^{(k)}\}_{k=0}^d$ are the coefficients determined by the expression of $r^{(d+1)}$ by $\{r^{(k)}\}_{k=0}^d$, and also
\[
r^{(d+l)} = \sum_{k=0}^d c^{(k)} r^{(k+l-1)}
= \sum_{k=0}^{l-1} (k+1) c^{(k)} r^{(k)} + l \sum_{k=l}^d c^{(k)} r^{(k)}
\]
for $l = 1, \dots, d$.
\end{remark}

\begin{proof}[Proof of Proposition \ref{prop:boxtimes-atom}]
It is sufficient to prove the case $\alpha = \beta = 1$, as the general case follows by dilation using $p_1(x) := p(x) \boxtimes_d (x-1/\alpha)^d$ and $q_1 (x) := q(x) \boxtimes_d (x-1/\beta)^d$.

Express $p(x) = \sum_{k=0}^d a_k r^{(k)}(x)$ and $q(x) = \sum_{k=0}^d b_k r^{(k)}(x)$ using Lemma \ref{lem:uniqueExpression}.
Since each multiplicity of $p$, $q$, $\{r^{(k)}\}_{k=0}^d$ at $1$ is $m^p_1$, $m^q_1$, and $d-k$, one has $a_d = a_{d-1} = \dots = a_{d-m^p_1+1} = 0 \ne a_{d-m^p_1}$ and $b_d = b_{d-1} = \dots = b_{d-m^q_1+1} = 0 \ne b_{d-m^q_1}$.
Hence,
\[
p = \sum_{k=0}^{d-m^p_1} a_k r^{(k)} \quad \text{and} \quad q(x) = \sum_{k=0}^{d-m^q_1} b_k r^{(k)}
\]
and thus
\[
p \boxtimes_d q = \sum_{k=0}^{2 d-m^p_1 - m^q_1} \left(\sum_{i+j=k} a_i b_j\right) r^{(k)}.
\]
Since $m = m_1^p + m_1^q - d \ge 0$, we have $2d - m_1^p - m_1^q = d-m \le d$.
Hence, the unique expansion of $p \boxtimes_d q$ in the basis $\{r^{(k)}\}_{k=0}^d$ is given by $\sum_{k=0}^{d-m} c_k r^{(k)}$, where $c_k = \sum_{i+j=k} a_i b_j$, cf. Remark \ref{rem:uniqueExpression}.

Finally, the coefficient of $r^{(d-m)}$ is $c_{d-m} = a_{d-m_1^p}b_{d-m_1^q} \ne 0$, which implies the multiplicity of $p \boxtimes_d q$ at $1$ is $m$.

\end{proof}

If $p \in \calP_d (\R)$ and $q \in \calP_d (\R_{\ge 0})$, their empirical root distributions $\meas{p}, \meas{q}$ and CDFs $\cdf_p, \cdf_q$ are well-defined.
For an atom triplet $(\alpha, \beta, \alpha \beta)$ of $(p, q, p \boxtimes_d q)$ with $\alpha\beta \ne 0$, Proposition \ref{prop:boxtimes-atom} implies
\begin{equation} \label{eq:atoms-finite-free-multiplicative}
\meas{p \boxtimes_d q}(\{\alpha \beta\}) = \meas{p}(\{\alpha\}) + \meas{q}(\{\beta\}) -1.
\end{equation}

Given a polynomial $p \in \calP_d (\R)$ and a root $\alpha$ of $p$, let $[\alpha]_{-}^p$ (resp. $[\alpha]_{+}^p$) denote the number of roots of $p$ strictly smaller (resp. strictly larger) than $\alpha$.
If $m^p_\alpha$ is the multiplicity of $p$ at $\alpha$, then $[\alpha]_{-}^p + m^p_\alpha +[\alpha]_{+}^p = d$,
\[  \cdf_p (\alpha) = \frac{[\alpha]_{-}^p + m^p_\alpha}{d}, \quad \text{and} \quad 1- \cdf_p (\alpha) = \frac{[\alpha]_{+}^p}{d}.\]

\begin{proposition}[{\cite[Theorem 1.6 (i)]{KS}}] \label{prop:atoms-finite-free-multiplicative}
For $p, q \in \calP_d (\R_{\ge 0})$, and an atom triplet $(\alpha, \beta, \alpha \beta)$ of $(p, q, p \boxtimes_d q)$ with $\alpha\beta \ne 0$, one has
\begin{equation} \label{eq:smallerRootsMultiplicative}
[\alpha \beta]^{p \boxtimes_d q}_{-} = [\alpha]^p_{-} + [\beta]^q_{-}.
\end{equation}
\end{proposition}
Note that Equation \eqref{eq:smallerRootsMultiplicative} is equivalent to
\begin{equation}\label{eq:atomsMultiplicative}
\cdf_{{p \boxtimes_d q}}(\alpha \beta) = \cdf_{{p}}(\alpha) + \cdf_{{q}}(\beta) - 1
\end{equation}
by Equation \eqref{eq:atoms-finite-free-multiplicative}.
This proposition can also be proven by the same approach taken in the proof of Proposition \ref{prop:cdf_at_atom_multiplicative}.

\begin{proposition}[{\cite[Theorem 1.6 (ii)]{KS}}] \label{prop:simple-finite-free-multiplicative}
For $p, q \in \calP_d (\R_{\ge 0})$, every non-trivial root of $p \boxtimes_d q$ is simple.
\end{proposition}

We omit the proof of this proposition as it was already proved in \cite{KS}. However, we will provide a precise proof for the analogous statement for $\boxplus_d$ in Appendix \ref{app:proof_prop_boxplus_simple} for self-completeness. The idea used there originates from Kostov and Shapiro \cite{KS}, so an interested reader can reconstruct the argument for $\boxtimes_d$ by imitating that proof.
As a quick remark, it suffices to prove this for $p, q \in \calP_d (\R_{> 0})$.
If, for instance, $q(x) = x^{d-k}q_{k}(x)$ where $q_{k}(0) \ne 0$, then $p \boxtimes_d q (x) = x^{d-k} \diff_{k|d}p \boxtimes_k q_k (x)$, and hence the (non)-trivial roots of $p \boxtimes_d q$ are exactly those of $\diff_{k|d}p \boxtimes_k q_k$, and vice versa.

\subsection{Finite free additive convolution}
\label{sec:atoms-additive}

First, we verify that an analogue of Proposition \ref{prop:boxtimes-atom} holds for $\boxplus_d$.
\begin{proposition} \label{prop:boxplus_atom}
Given two polynomials $p, q \in \calP_d (\C)$ such that $\alpha$ is a root of $p$ with multiplicity $m^p_\alpha$ and $\beta$ is a root of $q$ with multiplicity $m^q_\beta$. If $m := m^p_\alpha + m^q_\beta - d \ge 0$, then $\alpha+\beta$ is a root of $p \boxplus_d q$ of multiplicity $m$.
(If $m = 0$, $\alpha+\beta$ is not a root of $p\boxplus_d q$.)
\end{proposition}

\begin{proof}
By shifting $p_1(x) := p(x-\alpha)$ and $q_1(x) := q(x-\beta)$, we can assume $\alpha = \beta =0$ without loss of generality.
Since $(-1)^k \tilde{e}_k (p) = \diff_{k|d}(p)(0)$ for $k=0, \dots, d$,
we have $\tilde{e}_d (p) = \dots = \tilde{e}_{d- m^p_0 +1} (p) = 0 \ne \tilde{e}_{d- m^p_0} (p) $ and also $\tilde{e}_d (q) = \dots = \tilde{e}_{d- m^q_0 +1} (q) = 0 \ne \tilde{e}_{d- m^q_0} (q)$.
By definition, we have
\[
(-1)^k \diff_{k|d}(p \boxplus_d q)(0) = \tilde{e}_k (p \boxplus_d q) = \sum_{i+j = k}\binom{k}{i} \tilde{e}_i (p) \tilde{e}_j (q),
\]
and hence $\tilde{e}_{d-m} (p \boxplus_d q) = \binom{d-m}{d-m^p_0} \tilde{e}_{d-m^p_0} (p) \tilde{e}_{d-m^q_0} (q) \ne 0$ and $\tilde{e}_{k} (p \boxplus_d q) = 0$ for $d-m < k \le d$.
This implies that $p \boxplus_d q$ has a root at $0$ with multiplicity $m$.
\end{proof}

Analogous to the multiplicative case, if $m > 0$ in Proposition \ref{prop:boxplus_atom}, we call $\alpha + \beta$ a \textit{trivial root} of $p \boxplus_d q$, and $(\alpha, \beta, \alpha + \beta)$ an \textit{atom triplet} of $(p, q, p \boxplus_d q)$. Other roots are \textit{non-trivial}.

The proof of the following proposition is omitted as it closely mirrors that of Proposition \ref{prop:cdf_at_atom_additive}.
\begin{proposition} \label{prop:atoms-finite-free}
Let $p, q \in \calP_d (\R)$. For an atom triplet $(\alpha, \beta, \alpha + \beta)$ of $(p, q, p \boxplus_d q)$, we have
\[
\cdf_{{p \boxplus_d q}}(\alpha + \beta) = \cdf_{{p}}(\alpha) + \cdf_{{q}}(\beta) - 1.
\]
\end{proposition}

Finally, we conclude this section by stating that non-trivial roots of $p \boxplus_d q$ are simple, analogous to Proposition \ref{prop:simple-finite-free-multiplicative}.
\begin{proposition}\label{prop:boxplus_simple}
For $p, q \in \calP_d (\R)$, every non-trivial root of $p \boxplus_d q$ is simple.
\end{proposition}
The proof, nearly identical to that of \cite[Theorem 1.6 (ii)]{KS} (which is for the Schur-Szeg\H{o} composition), is provided in Appendix \ref{app:proof_prop_boxplus_simple} for completeness.

\section{Convergent sequence of polynomials} \label{sec:convergence}

This section investigates the convergence of sequences of finite free convolutions to their free counterparts, representing a core contribution of this paper.
We begin in Section \ref{sec:monotone} by establishing that the monotonicity of distances, known for free convolutions (cf. Proposition \ref{prop:monotonicity_distance_additive}), also holds for the finite versions $\boxplus_d$ and $\boxtimes_d$.
Building upon this crucial property, Section \ref{sec:approx_convolutions} extends prior convergence results (like Proposition \ref{prop:convergence-finite-free-additive}) by removing compactness assumptions on the underlying measures. This subsection further provides alternative proofs for existing theorems (e.g., Proposition \ref{prop:convergence-finite-free-multiplicative}), strengthens weak convergence to convergence in Kolmogorov distance, and demonstrates a slight generalization of Proposition \ref{prop:convergenceFiniteFreeMultiplicative}.

\subsection{Monotone property for distances}\label{sec:monotone}

We examine the relationship between Kolmogorov and Lévy distances for empirical measures of polynomials in $\calP_d (\R)$.
First, recall that for $p \in \calP_d(\R)$, its CDF, $\cdf_p = \cdf_{\meas{p}}$, is a step function with jumps of size $k/d$ for integers $k$. Specifically,
\[
\cdf_p: \R \to \left\{0, \frac{1}{d}, \frac{2}{d}, \dots, \frac{d}{d}\right\}.
\]
Consequently, for any $p, q \in \calP_d (\R)$, the Kolmogorov distance $d_K (p, q) = d_K(\meas{p}, \meas{q})$ must be a multiple of $1/d$:
\[
d_K (p, q) \in \left\{0, \frac{1}{d}, \frac{2}{d}, \dots, \frac{d}{d}\right\}.
\]
Furthermore, if $p \lessdot q$ (i.e., $p$ interlaces $q$), then their CDFs are ``close,'' leading to $d_K (p, q) \le 1/d$.

\begin{lemma} \label{lem:pan}
Let $p, q \in \calP_d (\R)$. If for some integer $0 \le l \le d$,
\begin{equation} \label{eq:pan_cond}
\cdf_q (x) \le \cdf_p (x) + \frac{l}{d} \qquad \text{for all } x \in \R,
\end{equation}
then there exists a sequence of polynomials $q = q^{(0)}, q^{(1)}, \dots, q^{(l)}$ such that $q^{(k-1)} \lessdot q^{(k)}$ for $k=1, \dots, l$, and $p \le q^{(l)}$.
As a consequence, for any $r \in \calP_d (\R)$,
\[
\cdf_{q \boxplus_d r} (x) \le \cdf_{p \boxplus_d r} (x) + \frac{l}{d} \qquad \text{for all } x \in \R,
\]
and for any $r \in \calP_d (\R_{\ge 0})$,
\[
\cdf_{q \boxtimes_d r} (x) \le \cdf_{p \boxtimes_d r} (x) + \frac{l}{d} \qquad \text{for all } x \in \R.
\]
\end{lemma}
\begin{proof}
Let $\lambda_1 (p) \le \lambda_2 (p) \le \dots \le \lambda_d (p)$ be the roots of $p$ and $(\lambda_i (q))_{i=1}^d$ be those of $q$.
By the assumption \eqref{eq:pan_cond}, we know that
\begin{equation} \label{eq:condition}
\lambda_{i} (p) \le \lambda_{l+i} (q)
\end{equation}
for $i = 1, \dots, d-l$.
Choose some $a > \max(\lambda_d (p), \lambda_d (q))$.
For $k = 1, \dots, l$, we inductively define
\[
q^{(k)} (x) := q^{(k-1)}(x) \cdot \frac{x-a}{x-\lambda_k (q)},
\]
that is, $q^{(k+1)}$ obtained by replacing the minimum root $\lambda_k(q)$ of $q^{(k)}$ with $a$.
This construction ensures $q^{(k-1)} \lessdot q^{(k)}$.
Besides,
\[
p(x) = \prod_{i=1}^d (x-\lambda_i(p)) \le (x-a)^l \prod_{i=1}^{d-l}(x-\lambda_{l+i}(q)) = q^{(l)}(x)
\]
by Equation \eqref{eq:condition}.

The rest follows from Propositions \ref{prop:InterlacingPreserving} and \ref{prop:OrderPreserving}.
\end{proof}

\begin{proof}[Proof of Theorem \ref{thm:main2}]
Let $p, q, r \in \calP_d (\R)$.
If $d_K (p, q) = l/d$ then
\[
\cdf_p (x) - \frac{l}{d} \le \cdf_q (x) \le \cdf_p (x)+ \frac{l}{d}
\]
for all $x \in \R$.
Hence, by Lemma \ref{lem:pan}
\[
\cdf_{p \boxplus_d r} (x) - \frac{l}{d} \le \cdf_{q \boxplus_d r} (x) \le \cdf_{p \boxplus_d r} (x)+ \frac{l}{d}
\]
for all $x \in \R$, which implies $d_K (p \boxplus_d r, q \boxplus_d r) \le l/d = d_K (p, q)$.

When $r \in \calP_d(\R_{\ge 0})$ the inequality $d_K (p \boxtimes_d r, q \boxtimes_d r) \le l/d = d_K (p, q)$ is proven by the same way.

For Lévy distance, it is enough to note that if
\[
\cdf_p(x-\epsilon)-\epsilon \le \cdf_{q} (x) \le \cdf_p(x+\epsilon)+\epsilon \quad \text{for all } x \in \R
\]
for some $\epsilon \in [\frac{l}{d}, \frac{l+1}{d})$ then
\[
\cdf_{\Shi_\epsilon p} (x) - \frac{l}{d} \le \cdf_{q} (x) \le \cdf_{\Shi_{-\epsilon} p} (x) + \frac{l}{d} \quad \text{for all } x \in \R.
\]
\end{proof}

In the proof of Theorem \ref{thm:main2}, we obtained the corresponding result for multiplicative convolution.
\begin{theorem} \label{thm:monotonicity-finite-free-multiplicative}
Let $p_1, p_2 \in \calP_d(\R)$ and $r \in \calP_d(\R_{\ge 0})$. Then,
$d_K (p_1 \boxtimes_d r, p_2 \boxtimes_d r) \le d_K (p_1, p_2)$.
\end{theorem}

\subsection{Convergence of finite free convolutions} \label{sec:approx_convolutions}

This subsection considers sequences of polynomials $p_d \in \calP_d (\R)$ whose empirical root distributions $\meas{p_d}$ converge to a probability measure $\mu \in \calM (\R)$ as $d \to \infty$.
We first confirm that any $\mu \in \calM(K)$ (for $K = \R, \R_{\ge 0}$, or $[a,b]$) can be approximated in Kolmogorov distance by the empirical measure of a polynomial $p_d \in \calP_d(K)$.

\begin{lemma} \label{lem:convergentSequenceOfPolynomials}
Let $K$ be $\R$, $\R_{\ge 0}$, or $[a,b]$ for $a,b \in \R$ with $a<b$.
For any $\mu \in \calM (K)$ and $d \in \N$, there exists a polynomial $p_d \in \calP_d (K)$ such that
$d_K (\meas{p_d}, \mu) \le 1/d$.
\end{lemma}
\begin{proof}
For $k =1, \dots, d-1$, define the quantile $\lambda_k = \inf\{x \in K \mid \cdf_\mu(x) \ge k/d \}$. 
Define $p_d(x) = (x-\lambda_{d-1})\prod_{k=1}^{d-1} (x-\lambda_k)$ and then we get $d_K(\meas{p_d}, \mu) \le 1/d$ and $p_d \in \calP_d (K)$ by the construction.
\end{proof}
Therefore, for any $\mu \in \calM (\R)$, we can construct a sequence of polynomials $(p_d)_{d=1}^\infty$ such that $\meas{p_d} \to \mu$ in Kolmogorov distance (and thus weakly).
In particular, if $d_K(\meas{p_d}, \mu) \to 0$ then $\meas{p_d}(\{x\}) \to \mu(\{x\})$ for every $x \in \R$ by Equation \eqref{eq:reflected_cdf_relation}.

\begin{proof}[Proof of Theorem \ref{thm:main3}]
Let $\epsilon>0$. Since the set of discontinuity points of $\cdf_\mu$ and $\cdf_\nu$ is countable, we can choose $a>0$ such that $\pm a$ are continuity points for both $\cdf_\mu$ and $\cdf_\nu$, and additionally $\max\{\cdf_{\mu}(-a), \cdf_\nu(-a), 1-\cdf_{\mu}(a), 1-\cdf_{\nu}(a)\} < \epsilon$.
This ensures $d_K(\mu, (\mu)_a)<\epsilon$ and $d_K(\nu, (\nu)_a) <\epsilon$, and hence
\begin{equation*}
d_K(\mu \boxplus \nu, (\mu)_a \boxplus (\nu)_a) < 2\epsilon
\end{equation*}
by Proposition \ref{prop:monotonicity_distance_additive}.

Since $\meas{p_d} \weakto \mu$ and $\meas{q_d} \weakto \nu$, and $\pm a$ are continuity points of $\cdf_\mu$ and $\cdf_\nu$, $\cdf_{p_d}(\pm a) \to \cdf_\mu(\pm a)$ and $\cdf_{q_d}(\pm a) \to \cdf_\nu(\pm a)$ as $d\to\infty$.
This implies $d_K({p_d}, ({p_d})_a) < \epsilon$ and $d_K({q_d}, ({q_d})_a) < \epsilon$, and hence
\begin{equation*}
d_K(p_d \boxplus_d q_d, (p_d)_a \boxplus_d (q_d)_a) < 2\epsilon
\end{equation*}
by Theorem \ref{thm:main2} for $d \in \N$ large enough.

The measures $(\mu)_a$ and $(\nu)_a$ have the supports contained in $[-a,a]$.
Also $\meas{(p_d)_a} \weakto (\mu)_a$ and $\meas{(q_d)_a} \weakto (\nu)_a$.
\begin{equation*}
d_L(\meas{(p_d)_a \boxplus_d (q_d)_a}, (\mu)_a \boxplus (\nu)_a) < \epsilon
\end{equation*}
by Proposition \ref{prop:convergence-finite-free-additive} for $d \in \N$ large enough.

Therefore, combining the inequalities above, we have
\[
d_L(\meas{p_d \boxplus_d q_d}, \mu \boxplus \nu) \le 5\epsilon
\]
for $d\in\N$ large enough.
Since $\epsilon>0$ was arbitrary, $d_L(\meas{p_d \boxplus_d q_d}, \mu \boxplus \nu) \to 0$, implying weak convergence.

For convergence in Kolmogorov distance, assume $d_K(\meas{p_d}, \mu) \to 0$ and $d_K(\meas{q_d}, \nu) \to 0$ as $d\to\infty$.
It is sufficient to prove the convergence
\[
\cdf_{p_d \boxplus_d q_d} (\gamma) \to \cdf_{\mu \boxplus \nu}(\gamma)
\]
for every non-continuous point $\gamma$ of $\mu \boxplus \nu$, which means $(\mu \boxplus \nu) (\{\gamma\})>0$.
By Proposition \ref{prop:atom_regularity_additive}, there exist $\alpha, \beta \in \R$ such that $\gamma = \alpha + \beta$ and
$(\mu \boxplus \nu) (\{\gamma\}) = \mu(\{\alpha\}) + \nu(\{\beta\}) - 1$.
Since, the empirical root distributions of $p_d$ and $q_d$ converge $\mu$ and $\nu$ in Kolmogorov distance, respectively, we have
\[
\meas{p_d}(\{\alpha\}) \to \mu(\{\alpha\}) \quad \text{and} \quad \meas{q_d}(\{\beta\}) \to \nu(\{\beta\}).
\]
Hence, for any $\epsilon \in (0, (\mu \boxplus \nu) (\{\gamma\})/2)$, there exists $d_0 \in \N$ such that $p_d$ (resp. $q_d$) has a root at $\alpha$ (resp. $\beta$) with multiplicity more than $d(\mu(\{\alpha\})- \epsilon)$ (resp. $d(\nu(\{\beta\})- \epsilon)$) for all $d> d_0$.
Thus, Proposition \ref{prop:boxplus_atom} implies that $p_d \boxplus_d q_d$ has a root $\gamma = \alpha + \beta$ with multiplicity more than $d(\mu(\{\alpha\}) + \nu(\{\beta\}) -1 - 2\epsilon) = d((\mu \boxplus \nu) (\{\gamma\}) -2\epsilon)>0$.
That is, $(\alpha, \beta, \gamma)$ is the atom triplet of $(p_d, q_d, p_d \boxplus_d q_d)$ for $d \in \N$ large enough.
Therefore,
\[
\cdf_{p_d \boxplus_d q_d} (\gamma) = \cdf_{p_d}(\alpha) + \cdf_{q_d}(\beta) -1 \to \cdf_\mu(\alpha) + \cdf_\nu(\beta) - 1 = \cdf_{\mu \boxplus \nu}(\gamma)
\]
as $d \to \infty$ by Propositions \ref{prop:cdf_at_atom_additive} and \ref{prop:atoms-finite-free}.
\end{proof}

A similar argument yields the multiplicative version, providing an alternative proof for Proposition \ref{prop:convergence-finite-free-multiplicative} and extending it to Kolmogorov distance.
\begin{theorem} \label{thm:convergence_multiplicative}
Let $(p_d)_{d\in\N}$ and $(q_d)_{d\in\N}$ be sequences of polynomials with $p_d, q_d \in \calP_d (\R_{\ge 0})$. Assume $\meas{p_d} \weakto \mu$ and $\meas{q_d} \weakto \nu$ for $\mu, \nu \in \calM(\R_{\ge 0})$ as $d \to \infty$.
Then, $\meas{p_d \boxtimes_d q_d} \weakto \mu \boxtimes \nu$.
Moreover, if $d_K(\meas{p_d}, \mu) \to 0$ and $d_K(\meas{q_d}, \nu) \to 0$, then $d_K(\meas{p_d \boxtimes_d q_d}, \mu \boxtimes \nu) \to 0$.
\end{theorem}

The same method also extends to the case where one measure is not restricted to $\R_{\ge 0}$, provided the other has compact support in $\R_{\ge 0}$.
\begin{theorem} \label{thm:convergence_mixed}
Given $\mu \in \calM(\R)$ and $\nu \in \calM([0, a])$ for some $a>0$.
If $(p_d)_{d\in\N}$ with $p_d \in \calP_d(\R)$ satisfies $\meas{p_d} \weakto \mu$, and $(q_d)_{d\in\N}$ with $q_d \in \calP_d([0,a])$ satisfies $\meas{q_d} \weakto \nu$, then
$\meas{p_d \boxtimes_d q_d} \weakto \mu \boxtimes \nu$.
\end{theorem}
The Kolmogorov distance result might also hold but would require ensuring that Proposition \ref{prop:monotonicity_distance_multiplicative_kolmogorov} or Theorem \ref{thm:monotonicity-finite-free-multiplicative} applies appropriately when $p_1, p_2$ are not restricted to $\R_{\ge 0}$.

\section*{Final remarks}

This paper has advanced the understanding of finite free convolutions by extending the convergence of $\boxplus_d$ and $\boxtimes_d$ to their respective infinite counterparts, $\boxplus$ and $\boxtimes$, notably without requiring prior compactness assumptions on the underlying measures.
A significant outcome is the demonstration that improved initial polynomial approximations, specifically in the Kolmogorov distance, translate directly to enhanced convergence for the resulting convolutions.
Despite these advances, the most general scenario for the finite free multiplicative convolution---involving measures $\mu \in \calM(\R)$ and $\nu \in \calM(\R_{\ge 0})$---remains open, with questions persisting about both weak convergence and the possibility of strengthening this to Kolmogorov distance.

More broadly, this research has sought to illuminate the fruitful interplay between finite free probability (often approached via the geometry of polynomials) and the established theory of (infinite) free probability.
By leveraging perspectives from one domain to address problems in the other, we have shown that solutions can often be developed through analogous methodologies.
This synergy is not unidirectional; insights from polynomial theory also provide valuable feedback to free probability.
A case in point is the role of Propositions \ref{prop:cdf_at_atom_additive} and \ref{prop:cdf_at_atom_multiplicative} in the proof of Theorem \ref{thm:main1}. These propositions, which are critical for understanding atomic structure, build upon Kostov and Shapiro's work on the Schur--Szegö composition \cite{KS} and translate findings from polynomial theory into the language of free probability---a connection that, despite its directness, appears to be a novel observation in this context.
It is hoped that this paper will stimulate further collaborative advancements at the interface of these two fields.

\section*{Acknowledgements}

This research was supported by the JSPS Open Partnership Joint Research Projects Grant (No. JPJSBP120209921) and JSPS Research Fellowship for Young Scientists PD (KAKENHI Grant No. 24KJ1318).
Significant portions of this research were conducted during the author's stay to CIMAT.
The author expresses gratitude to Octavio Arizmendi for valuable discussions during the initial stages of this work.


\appendix

\section{Proof of Proposition \ref{prop:boxplus_simple}} \label{app:proof_prop_boxplus_simple}

The argument relies on the continuity of polynomial roots with respect to coefficients. Specifically:
\begin{enumerate}
\item If $p \in \calP_d (\R)$ has only simple roots, then $p(x)+\epsilon$ remains in $\calP_d (\R)$ (i.e., retains $d$ real roots, which will also be simple) for all sufficiently small $\epsilon \in \R$.
\item If $p \in \calP_d (\R)$ and $\alpha$ is a root of $p$ with multiplicity $2$, then there exists $\epsilon_0 >0$ such that
$p + \epsilon \notin \calP_d (\R)$ for all $0<\epsilon< \epsilon_0$ or
$p - \epsilon \notin \calP_d (\R)$ for all $0<\epsilon< \epsilon_0$.
\end{enumerate}

\begin{lemma} \label{lem:appendix_basis}
If $p(x) = \prod_{i=1}^d (x-\alpha_i)$ where $\alpha_i \in \C$ are distinct roots. Let $p_j (x) := p(x)/(x-\alpha_j)$ for $j \in \{1, \dots, d\}$. Then, the set of $d+1$ polynomials $\{p(x), p_1(x), \dots, p_d(x)\}$ forms a basis for the vector space $\C_d [x]$.
\end{lemma}

\begin{proof}
Since $\dim(\C_d[x]) = d+1$, it suffices to prove linear independence. Assume
$c_0 p(x) + \sum_{j=1}^d c_j p_j (x) = 0$ for some $c_j \in \C$.
Substituting $x=\alpha_k$ for any $k \in \{1, \dots, d\}$:
Since $p(\alpha_k)=0$ and $p_j(\alpha_k)=0$ for $j \ne k$, the equation becomes $c_k p_k(\alpha_k) = 0$.
As $\alpha_k$ are distinct, $p_k(\alpha_k) = \prod_{i \ne k}(\alpha_k - \alpha_i) \ne 0$. Thus, $c_k=0$ for all $k=1, \dots, d$.
The equation then reduces to $c_0 p(x) = 0$. Since $p(x)$ is not identically zero, $c_0=0$.
Hence, all coefficients are zero, proving linear independence.
\end{proof}

\begin{proof}[Proof of Proposition \ref{prop:boxplus_simple}]
The proof is by contradiction.
Assume $p \boxplus_d q$ has a nontrivial root at $0$ by using the shift.
After derivation, we may also assume the multiplicity is $2$.
We will do the case-by-case arguments.

First, suppose that every root of $p$ is simple.
Then, for small enough $\epsilon \in \R$, one has $p + \epsilon \in \calP_d (\R)$.
Now, take $\epsilon_0 \in \R$ such that $p \boxplus_d q + \epsilon_0 \notin \calP_d (\R)$, but $(p+\epsilon_0) \boxplus_d q = p \boxplus_d q + \epsilon_0$ by Proposition \ref{prop:preserve}.
This is a contradiction.

Second, let us assume $p$ has a root $\alpha$ with multiplicity $m^p_\alpha > 1$.
Then, define $p_\alpha (x) = p(x)/(x-\alpha)$ and consider $p^\delta (x) = p(x) - \delta p_\alpha (x)$ for $\delta \in \R$.
Note that $p^\delta$ has a root $\alpha$ with multiplicity $m^p_\alpha -1$ and instead a new root $\alpha + \delta$.
Then,
\[  p^\delta \boxplus_d q = p \boxplus_d q - \delta p_\alpha \boxplus_d q \in \calP_d (\R).\]
Note that $p_\alpha \in \calP_{d-1}(\R)$ and $U_\alpha (x) := p_\alpha \boxplus_d q = p_\alpha \boxplus_{d-1} \diff_{(d-1)|d} q \in \calP_{d-1} (\R)$. 
If $U_\alpha (0) \ne 0$, we can take a small $\delta \in \R$ and then
$p^\delta \boxplus_d q \notin \calP_d (\R)$, so a contradiction.
Hence, we should consider the case $U_\alpha (0) = 0$.
There are two cases, $U'_\alpha (0) = 0$ or $U'_\alpha (0) \ne 0$.

If $U_\alpha (0) = 0$ and $U'_\alpha (0) = 0$ then we may take $\delta \in \R$ as $\alpha + \delta$ is a root of $p^\delta$ with multiplicity $1$.
Besides, $p^\delta$ has a root at $0$ with multiplicity more than $1$ and the multiplicity at $\alpha$ decreases.

The case $U_\alpha (0) = 0$ and $U'_\alpha (0) \ne 0$ is a problem.
We will see that if there is another root $\beta \ne \alpha$ of $p$ such that $U_\beta (0) = 0$ and $U'_\beta (0) \ne 0$ then we may decrease the multiplicity of $p$ at $\alpha$ by the similar argument.
Let $p_{(\alpha, \beta)}(x) = \frac{p(x)}{ (x-\alpha)(x-\beta)}$.
Then
\[  p^{(\delta, \epsilon)}(x) = p(x) - \delta p_\alpha (x) - \epsilon p_\beta (x) + \delta \epsilon p_{(\alpha, \beta)}(x)\]
for $\delta, \epsilon \in \R$.
Note that $p_{(\alpha, \beta)} = \frac{p_\alpha - p_\beta}{\alpha - \beta}$ and hence $p_{(\alpha , \beta)} \boxplus_d q (0) = 0$.
If $\epsilon$ is taken as sufficiently small then $((p_\alpha - \epsilon p_{(\alpha, \beta)}) \boxplus_d q)'(0) = U'_\alpha (0) - \epsilon (p_{(\alpha, \beta)} \boxplus_d q)'(0) \ne 0$.

Also, setting $\delta = - \epsilon \frac{U'_\beta (0)}{U'_\alpha (0) - \epsilon (p_{(\alpha, \beta)} \boxplus_d q)'(0)}$ implies
$(p^{(\delta, \epsilon)} \boxplus_d q)'(0) = 0$.
Hence, $p^{(\delta, \epsilon)}$ has a root at $0$ with multiplicity more than $1$ and the multiplicity of every root of $p^{(\delta, \epsilon)}$ decreases.

Finally, the worst case is as follows: $p$ is not simple, and for every root $\alpha$ with multiplicity $m^p_\alpha > 1$ we have $U_\alpha (0) = 0$ and $U'_\alpha (0) \ne 0$.
By the discussion above, such a root is only one, just call $\alpha$ with multiplicity $m^p_\alpha>1$ and the other roots $\{\alpha_i\}_{i=1}^{d-m^p_\alpha}$.
Also, for any other root $\alpha_i$ of $p$, we have $U_{\alpha_i}(0) = U'_{\alpha_i}(0) = 0$.
It is same in $q$ as well, so define the root $\beta$ with multiplicity $m^q_\beta > 1$ and the other roots $\{\beta_i\}_{i=1}^{d-m^q_\beta}$.
By lemma, a linear combination of $p$ and $\{p_{\alpha_i}\}_{i=1}^{d-m^p_\alpha}$ expresses $(x-\alpha)^d$ and hence $(x-\alpha)^d \boxplus_d q (x) = q (x-\alpha)$ is a linear combination of $p \boxplus_d q$ and $\{p_{\alpha_i} \boxplus_d q = U_{\alpha_i}\}_{i=1}^{d-m^p_\alpha}$.
Thus, $q(x-\alpha)$ has a root at $0$ with multiplicity more than $1$.
It means $\beta = - \alpha$.
Since $0$ is a non-trivial root of $p \boxplus_d q$, $m^p_\alpha + m^q_\beta -d <0$ by Proposition \ref{prop:boxplus_atom}.
By lemma again, a linear combination of $p$ and $\{p_{\alpha_i}\}_{i=1}^{d-m^p_\alpha}$ expresses $r(x) = (x-1-\alpha)^{m^q_\beta}(x-\alpha)^{d - m^q_\beta}$ and hence $r \boxplus_d q$ has a root at $0$ with multiplicity more than $1$. This is a contradiction by Proposition \ref{prop:boxplus_atom}.
\end{proof}




\vspace{6mm}
\begin{enumerate}
\item[]
    \hspace{-10mm} Katsunori Fujie\\
    Department of Mathematics, Kyoto University.\\
    Kitashirakawa, Oiwake-cho, Sakyo-ku, Kyoto, 606-8502, Japan.\\
    email: fujie.katsunori.42m@st.kyoto-u.ac.jp\\
    URL: \url{https://sites.google.com/view/katsunorifujie}
\end{enumerate}

\end{document}